\documentclass{article}
\usepackage[english]{babel}
\usepackage{amsmath}
\usepackage{amsthm}
\usepackage{amssymb}
\usepackage{amscd}
\usepackage[all]{xy}
\usepackage[latin1]{inputenc}

\newtheorem*{lemas}{Lemma}
\newtheorem*{teors}{Theorem}

\begin{document}

\title{Addendum to 'Direct limits in the heart of a t-structure: The case of a torsion pair'}

\author{Carlos E. Parra 
\\
Departamento de Matem\'aticas\\ Universidad de los Andes \\ ({\bf 5101}) M\'erida\\ VENEZUELA\\
{\it carlosparra@ula.ve} \\  \\ Manuel Saor\'in 
\\ Departamento de Matem\'aticas\\
Universidad de Murcia, Aptdo. 4021\\
30100 Espinardo, Murcia\\
SPAIN\\ {\it msaorinc@um.es} }


\date{}

\maketitle


\begin{abstract}

{\bf Let $\mathcal{G}$ be a Grothendieck category, let $\mathbf{t}=(\mathcal{T},\mathcal{F})$ be a torsion pair in $\mathcal{G}$ and let $(\mathcal{U}_\mathbf{t},\mathcal{W}_\mathbf{t})$ be the associated Happel-Reiten-Smal$\o$ t-structure in the derived category $\mathcal{D}(\mathcal{G})$. We prove that the heart of this t-structure is a Grothendieck category if, and only if, the torsionfree class $\mathcal{F}$ is closed under taking direct limits in $\mathcal{G}$. 
 }
\end{abstract}

The present short note gives a definite answer to a question left open in \cite{PS}. All throughout the paper, let $\mathcal{G}$ be a Grothendieck category and $\mathbf{t}=(\mathcal{T},\mathcal{F})$ be a torsion pair in $\mathcal{G}$. The Happel-Reiten-Smal$\o$ t-structure in the (unbounded) derived category $\mathcal{D}(\mathcal{G})$ associated to $\mathbf{t}$ is the pair $(\mathcal{U}_\mathbf{t},\mathcal{U}_\mathbf{t}^{\perp}[1])=(\mathcal{U}_\mathbf{t},\mathcal{W}_\mathbf{t})$
in $\mathcal{D}(\mathcal{G})$, where:

\begin{center}
$\mathcal{U}_\mathbf{t}=\{X\in\mathcal{D}^{\leq 0}(\mathcal{G}):$
$H^0(X)\in\mathcal{T}\}$

$\mathcal{W}_\mathbf{t}=\{Y\in\mathcal{D}^{\geq -1}(\mathcal{G}):$
$H^{-1}(Y)\in\mathcal{F}\}$.
\end{center}
(see \cite{HRS}).

In such a case the heart $\mathcal{H}_\mathbf{t}=\mathcal{U}_\mathbf{t}\cap\mathcal{W}_\mathbf{t}$ of the t-structure, which is an abelian category by \cite{BBD}, consists of the complexes $M\in\mathcal{D}(\mathcal{G})$ such that $H^j(M)=0$, for $j\neq -1,0$, that $H^{-1}(M)\in\mathcal{F}$ and that $H^0(M)\in\mathcal{T}$.  One of the main results in \cite{PS}, its Theorem 4.8, implies that if $\mathcal{H}_\mathbf{t}$  is a Grothendieck category, then the torsionfree class $\mathcal{F}$ is closed under taking direct limits in $\mathcal{G}$. The main result of this note says that the converse is also true, thus anwering in the affirmative Question 4.11 in the mentioned paper and fulfilling the task started in earlier papers by Colpi, Gregorio, Mantese and Tonolo (see  \cite{CGM}, \cite{CMT} and \cite{CG}).

The reader is referred to \cite{PS} and the references therein for all the terminology and notation used in this paper. The crucial auxiliary result is the following.

\begin{lemas}[\bf{1.1}] \label{lem.representation-objects-heart}
Each object  of $\mathcal{H}_\mathbf{t}$ can be represented by a complex $N:\cdots \longrightarrow 0 \longrightarrow N^{-1} \longrightarrow N^{0} \longrightarrow 0 \longrightarrow \cdots $ 
, concentrated in degrees $-1$ and $0$, such that $N^{-1}$ is an injective object of $\mathcal{G}$. Suppose that $N$ is such a complex and let $M: \cdots \longrightarrow 0\longrightarrow M^{-1}\longrightarrow M^0\longrightarrow 0 \longrightarrow \cdots$ be a complex concentrated in degrees $-1,0$ which is in $\mathcal{H}_\mathbf{t}$. Then the canonical map

\begin{center}
$\text{Hom}_{\mathcal{K}(\mathcal{G})}(M,N)\longrightarrow\text{Hom}_{\mathcal{D}(\mathcal{G})}(M,N)=\text{Hom}_{\mathcal{H}_\mathbf{t}}(M,N)$
\end{center}
is bijective.
\end{lemas}
\begin{proof}
Each object of $\mathcal{H}_\mathbf{t}$ is a quasi-isomorphic to a complex of injective objects 

\begin{center}
$E: \cdots \longrightarrow 0\longrightarrow E^{-1}\stackrel{d^{-1}}{\longrightarrow}E^{0}\stackrel{d^{0}}{\longrightarrow} \cdots \stackrel{\hspace{0.3 cm}d^{n-1}}{\longrightarrow}E^{n}\stackrel{d^{n}}{\longrightarrow}  \cdots $
\end{center}
We then take the 'intelligent' truncation $\tau^{\leq 0}E: \cdots \longrightarrow 0\longrightarrow E^{-1}\longrightarrow Z^0\longrightarrow 0 \longrightarrow \cdots$, where $Z^0=\text{Ker}(d^0)$, which is obviously isomorphic to $E$ in $\mathcal{D}(\mathcal{G})$.

Let now $M: \cdots \longrightarrow 0\longrightarrow M^{-1}\stackrel{d_M}{\longrightarrow}M^0\longrightarrow 0 \longrightarrow \cdots$ and
$N: \cdots \longrightarrow 0\longrightarrow N^{-1}\stackrel{d_N}{\longrightarrow}N^0\longrightarrow 0 \longrightarrow \cdots$ be two complexes concentrated in degrees $-1$ and $0$ which are in $\mathcal{H}_\mathbf{t}$, such that $N^{-1}$ is an injective object of $\mathcal{G}$. Let $0\rightarrow N^0\longrightarrow E^0(N^0)\longrightarrow E^1(N^0)\longrightarrow  \cdots $ be the minimal injective resolution of $N^0$ in $\mathcal{G}$. The following commutative diagram gives a monomorphism $\iota:N\longrightarrow E$ in $\mathcal{C}(\mathcal{G})$, which is a quasi-isomorphism,  where $E$ is the complex given by the lower row of the diagram:

$$\xymatrix{\cdots \ar[r] & 0 \ar[r] & N^{-1} \ar[r] \ar@{=}[d] & N^{0} \ar[r] \ar@{^(->}[d] & 0 \ar[r] \ar[d]& \cdots \\ \cdots \ar[r] & 0 \ar[r] & N^{-1} \ar[r] & E^{0}(N^{0}) \ar[r] & E^{1}(N^{0}) \ar[r] & \cdots   }$$

We then have an exact sequence $0\rightarrow N\stackrel{\iota}{\longrightarrow}E\longrightarrow E/N\rightarrow 0$ in $\mathcal{C}(\mathcal{G})$, where $E/N$ is acyclic and concentrated in degrees $\geq 0$. It follows that $\text{Hom}_{\mathcal{C}(\mathcal{G})}(M,E/N)\cong\text{Hom}_{\mathcal{C}(\mathcal{G})}(H^0(M)[0],E/N)=0$. As a consequence, we get an isomorphism $\text{Hom}_{\mathcal{C}(\mathcal{G})}(M,N)\cong\text{Hom}_{\mathcal{C}(\mathcal{G})}(M,E)$, which in turn yields an isomorphism $\text{Hom}_{\mathcal{K}(\mathcal{G})}(M,N)\cong\text{Hom}_{\mathcal{K}(\mathcal{G})}(M,E)$. But $\text{Hom}_{\mathcal{K}(\mathcal{G})}(M,E)\cong\text{Hom}_{\mathcal{D}(\mathcal{G})}(M,N)=\text{Hom}_{\mathcal{H}_\mathbf{t}}(M,N)$ since $\iota$ is a homotopically injective resolution of $N$.
\end{proof}

We can now prove our desired result.

\begin{teors}[{\bf 1.2}] \label{teor.HRS}
Let $\mathcal{G}$ be a Grothendieck category, let $\mathbf{t}=(\mathcal{T},\mathcal{F})$ be a torsion pair in $\mathcal{G}$ and let $\mathcal{H}_\mathbf{t}$ be the heart of the associated t-structure in $\mathcal{D}(\mathcal{G})$. The following assertions are equivalent:

\begin{enumerate}
\item $\mathcal{H}_\mathbf{t}$ is a Grothendieck category;
\item $\mathcal{F}$ is closed under taking direct limits in $\mathcal{G}$.
\end{enumerate}
\end{teors}
\begin{proof}
By \cite[Theorem 4.8]{PS}, we just need to prove $2)\Longrightarrow 1)$ and, for that, it is enough to check  that the functor $H^{-1}:\mathcal{H}_\mathbf{t}\longrightarrow\mathcal{G}$ preserves direct limits. According to \cite[Corollary 1.7 and subsequent Remark]{AR}, we just need to consider direct systems $(M_\alpha )_{\alpha <\lambda}$ in $\mathcal{H}_\mathbf{t}$ such that $\lambda$ is a limit ordinal and, for each limit ordinal $\alpha <\lambda$, one has $M_\alpha =\varinjlim_{\beta <\alpha}M_\beta$. By the previous lemma, we can and shall assume that $M_\alpha$ is a complex $\cdots \longrightarrow 0\longrightarrow M_\alpha^{-1}\longrightarrow M_\alpha^0\longrightarrow 0 \longrightarrow \cdots$ concentrated in degrees $-1,0$ such that $M_\alpha^{-1}$ is an injective object of $\mathcal{G}$, for each $\alpha <\lambda$. The lemma also allows to assume that the connecting morphism $f_\alpha :M_\alpha\longrightarrow M_{\alpha +1}$ is a chain map, so that $(M_\alpha )_{\alpha <\lambda}$ is also a direct system in $\mathcal{K}(\mathcal{G})$.

For each category $\mathcal{C}$, let us denote by $[\lambda ,\mathcal{C}]$ the category of $\lambda$-direct systems in  $\mathcal{C}$. We shall construct a direct system $(\tilde{M}_\alpha )_{\alpha <\lambda}$ in $\mathcal{C}(\mathcal{G})$ satisfying the following conditions:

\begin{enumerate}
\item[a)] Each $\tilde{M}_\alpha$ is a complex concentrated in degrees $-1$ and $0$;
\item[b)] If $p:\mathcal{C}(\mathcal{G})\longrightarrow\mathcal{D}(\mathcal{G})$ is the canonical functor, then the induced functor $p_*:[\lambda ,\mathcal{C}(\mathcal{G})]\longrightarrow [\lambda ,\mathcal{D}(\mathcal{G})]$ takes $(\tilde{M}_\alpha )_{\alpha <\lambda}$ to $(M_\alpha )_{\alpha <\lambda}$.
\end{enumerate}
We construct the $\tilde{M}_\alpha$ and the connecting maps $\tilde{f}_\alpha :\tilde{M}_\alpha\longrightarrow\tilde{M}_{\alpha +1}$ by transfinite induction on $\alpha <\lambda$. For $\alpha$ nonlimit, we just put $\tilde{M}_\alpha =M_\alpha$. Suppose now that $\alpha$ is any ordinal for which $\tilde{M}_\beta$ has already being defined, for all $\beta\leq\alpha$ ($\alpha$ included!). In particular, we have $\tilde{M}_{\alpha +1}=M_{\alpha +1}$ and $\tilde{M}_\alpha$ is isomorphic to $M_\alpha$ in $\mathcal{D}(\mathcal{G})$. By Lemma (1.1)
, there is a chain map $\tilde{f}_\alpha :\tilde{M}_\alpha\longrightarrow\tilde{M}_{\alpha +1}$ whose image by the canonical map $\text{Hom}_{\mathcal{C}(\mathcal{G})}(\tilde{M}_\alpha ,\tilde{M}_{\alpha +1})\longrightarrow \text{Hom}_{\mathcal{D}(\mathcal{G})}(\tilde{M}_\alpha ,\tilde{M}_{\alpha +1})=\text{Hom}_{\mathcal{D}(\mathcal{G})}(\tilde{M}_\alpha ,M_{\alpha +1})\cong\text{Hom}_{\mathcal{D}(\mathcal{G})}(M_\alpha ,M_{\alpha +1})$ is the  morphism $f_\alpha :M_\alpha\longrightarrow M_{\alpha +1}$. Finally, if $\alpha <\lambda$ is a limit ordinal, we put $\tilde{M}_\alpha =\varinjlim_{\beta <\alpha}\tilde{M}_\beta$, where the last direct limit is taken in $\mathcal{C}(\mathcal{G})$. Note that, also in this limit case, $\tilde{M}_\alpha$ is a complex concentrated in degrees $-1$ and $0$. Moreover, by \cite[Lemma 4.4]{PS} and transfinite induction, we know that, when viewed as an object of $\mathcal{D}(\mathcal{G})$, the complex $\tilde{M}_\alpha$ is in $\mathcal{H}_\mathbf{t}$ and is the direct limit in this latter category of all the $\tilde{M}_\beta$, with $\beta <\alpha$.

The just constructed direct system $(\tilde{M}_\alpha )_{\alpha <\lambda}$ obviously satisfies the requirements a) and b) above. Again by \cite[Lemma 4.4]{PS}, we know that we have a canonical morphism  $\varinjlim_{\mathcal{H}_\mathbf{t}}M_\alpha\cong\varinjlim_{\mathcal{H}_\mathbf{t}}\tilde{M}_\alpha\longrightarrow\varinjlim_{\mathcal{C}(\mathcal{G})}\tilde{M}_\alpha$ in $\mathcal{D}(\mathcal{G})$ which is an isomorphism in this latter category. We then  get an isomorphism $H^{-1}(\varinjlim_{\mathcal{H}_\mathbf{t}}M_\alpha)\cong H^{-1}(\varinjlim_{\mathcal{C}(\mathcal{G})}\tilde{M}_\alpha)\cong\varinjlim H^{-1}(\tilde{M}_\alpha )\cong\varinjlim H^{-1}(M_\alpha )$ which is inverse of the canonical morphism $\varinjlim H^{-1}(M_\alpha )\longrightarrow H^{-1}(\varinjlim_{\mathcal{H}_\mathbf{t}}M_\alpha )$.
\end{proof}


\begin{thebibliography}{9999}

\bibitem{AR} {\sc AD\'AMEK, J.; ROSICK\'Y, J.}: Locally presentable and accessible categories. London Math. Soc. Lect. Note Ser., vol. \textbf{189}. Cambridge University Press (1994).


\bibitem{BBD} {\sc BEILINSON, A.; BERNSTEIN, J.; DELIGNE, P.}: 'Faisceaux pervers'. Astèrisque \textbf{100}, Soc. Math. France, Paris (1982), 5-171.



\bibitem{CG} {\sc COLPI, R.; GREGORIO, E.}: The heart of a cotilting torsion pair is a Grothendieck category. Preprint.

\bibitem{CGM} {\sc COLPI, R.; GREGORIO, E.; MANTESE, F.}:  On the heart of a faithful torsion pair. J. Algebra \textbf{307} (2007), 841-863.

\bibitem{CMT} {\sc COLPI, R.; MANTESE, F.; TONOLO, A.}:  When the heart of a faithful torsion pair is a module category. J. Pure and Appl. Algebra \textbf{215} (2011), 2923-2936.

\bibitem{HRS} {\sc HAPPEL, D.; REITEN, I.; SMALO, S.O.}: Tilting in abelian categories and quasitilted algebras. Memoirs AMS, vol. \textbf{120} (1996). 

\bibitem{PS} {\sc PARRA, C.; SAORIN, M.}: Direct limits in the heart of a t-structure: the case of a torsion pair. J. Pure and Appl. Algebra \textbf{219}(9) (2015), 4117-4143.






\end{thebibliography}
\end{document}